\documentclass[12pt]{article}

\usepackage{nicefrac}
\usepackage{amsfonts,amssymb,amsmath,amsthm,bbm}
\usepackage{epsfig}
\usepackage{xcolor}
\usepackage{enumerate}

\newcommand\Z{\mathbb{Z}}
\newcommand\cP{\mathcal{P}}
\renewcommand\d{\partial}
\renewcommand\P{\mathbb{P}}

\theoremstyle{plain}
\newtheorem{theorem}{Theorem}[section]
\newtheorem{proposition}[theorem]{Proposition}

\theoremstyle{definition}

\theoremstyle{remark}
\newtheorem{remark}{Remark}[section]

\begin{document}

\title{Transcritical bifurcation for the conditional distribution of a diffusion process}
\author{Michel Bena\"im$^1$, Nicolas Champagnat$^{2}$, William O\c cafrain$^{2}$,  Denis Villemonais$^{2}$}
\footnotetext[1]{Institut de Math\'{e}matiques, Universit\'{e} de Neuch\^{a}tel, Switzerland.}
\footnotetext[2]{Université de Lorraine, CNRS, Inria, IECL, F-54000 Nancy, France.}

\maketitle

\begin{abstract}
    In this article, we describe a simple class of models of absorbed diffusion processes with parameter, whose conditional law exhibits a transcritical bifurcation. Our proofs are based on the description of the set of quasi-stationary distributions for general two-clusters reducible processes.
\end{abstract}

\section{Model, motivation and main result}

Let $D=(0,5)$ and consider Lipschitz functions $\varphi^1,\varphi^2,\psi^1,\psi^2:D\to[0,1]$ such that $\psi_1+\psi_2=1$ and (see Figure~\ref{fig:graphs})
\begin{align*}
    \begin{cases}
    \varphi^1_{\rvert (0,1]}\equiv 1,\ \varphi^1_{\rvert [1,2]}\leq 1,\ \varphi^1_{\rvert [2,5)}\equiv 0,\\
    \varphi^2_{\rvert (0,3]}\equiv 0,\ \varphi^2_{\rvert [3,4]}\leq 1,\ \varphi^2_{\rvert [4,5)}\equiv 1,\\
    \psi^1_{\rvert (0,2]}\equiv 1,\ 0<\psi^1_{\rvert  (2,3)}\leq 1,\ \psi^1_{\rvert [3,5)}\equiv 0,\\
    \psi^2_{\rvert (0,2]}\equiv 0,\ 0<\psi^2_{\rvert  (2,3)}\leq 1,\ \psi^2_{\rvert [3,5)}\equiv 1.
    \end{cases}
\end{align*}

\begin{figure}
{%\small
	\def\svgwidth{\linewidth}
	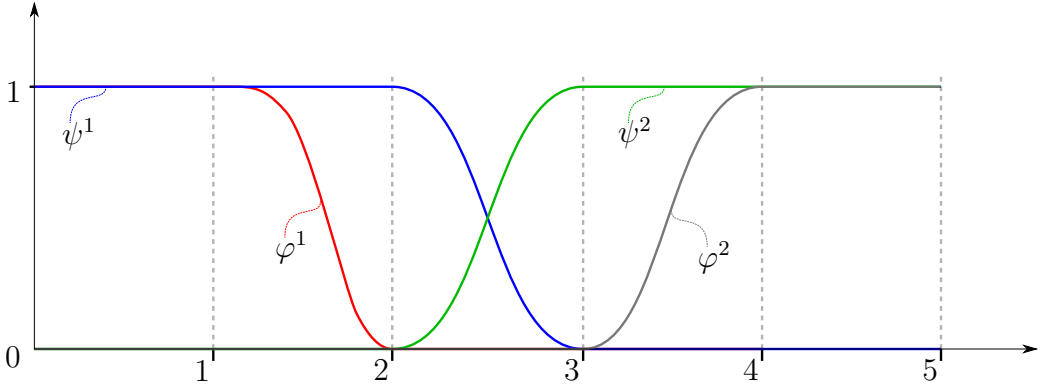}
    \caption{\label{fig:graphs} Representation of the graphs of $\varphi^1,\varphi^2,\psi^1$ and $\psi^2$.}
\end{figure}

For all $\alpha>0$, we consider the absorbed diffusion process $X^\alpha$ evolving according to the It\^o's stochastic differential equation
\begin{align}
\label{eq:EDS}
\mathrm dX^\alpha_t=\left(\varphi^1(X^\alpha_t)+\sqrt{\alpha}\,\varphi^2(X^\alpha_t)\right)\,\mathrm dB_t\,+(\psi^1(X_t)+\alpha\,\psi^2(X_t))\,\mathrm dt,
\end{align}
stopped when it reaches $\partial D=\{0,5\}$, and where $B$ is a standard one dimensional Brownian motion. This defines a sub-Markov semi-group $(P^\alpha_t)_{t\in\mathbb R_+}$ on $D$ by
\begin{align*}
\delta_x P^\alpha_t f=\mathbb E_x(f(X^\alpha_t) 1_{X^\alpha_t\in D}),\ \forall f\in L^\infty(D)
\end{align*}
and a semi-flow $(\Phi^\alpha_t)_{t\in\mathbb R_+}$ on the set $\mathcal P(D)$ of probability measures on $D$ by
\begin{align*}
\Phi^\alpha_t(\mu)=\mathbb P_\mu(X_t\in\cdot\mid X_t\in D).
\end{align*}
Observe that the diffusion coefficient in~\eqref{eq:EDS} vanishes on the interval [2,3]. Since in addition the drift
  coefficient is positive on $[2,3]$ for any $\alpha>0$, the set $[3,5]$ is absorbing and $\mathbb{P}_2(\forall
    t\geq 0, X_t\geq 2,\ \exists t\geq 0, X_t=3)=1$.

Hence, the family of diffusion processes $(X^\alpha)_{\alpha \geq 0}$ has some similarities with the family of
  discrete-time Markov chains $(X^{a,b})_{a,b \in (0,1)}$, where $X^{a,b}$ is defined on $\{1,2,\d\}$, absorbed at
  $\d$, with transition submatrix on $\{1,2\}$ given by
$$\begin{pmatrix}
a & 1-a \\
0 & b \end{pmatrix}.$$
It is pointed in \cite[Example 3.5]{BenaimCloezEtAl2016} that the probability measures $\nu_2 := \delta_2$ and $\nu := \frac{a-b}{1-b} \delta_1 + \frac{1-a}{1-b} \delta_2$ (when $a > b$) are such that 
\begin{itemize}
    \item If $a > b$, $\lim_{n \to \infty} \P_\mu(X^{a,b}_n = i | X^{a,b}_n \ne \d) = \nu(\{i\})$ for all $i=1,2$ and $\mu \ne \delta_2$.
    \item Otherwise, $\lim_{n \to \infty} \P_\mu(X^{a,b}_n = i | X^{a,b}_n \ne \d) = \nu_2(\{i\})$ for all $i=1,2$ and $\mu \in \cP(\{1,2\})$.
\end{itemize}

\begin{figure}
    {%\small
        \def\svgwidth{\linewidth}
        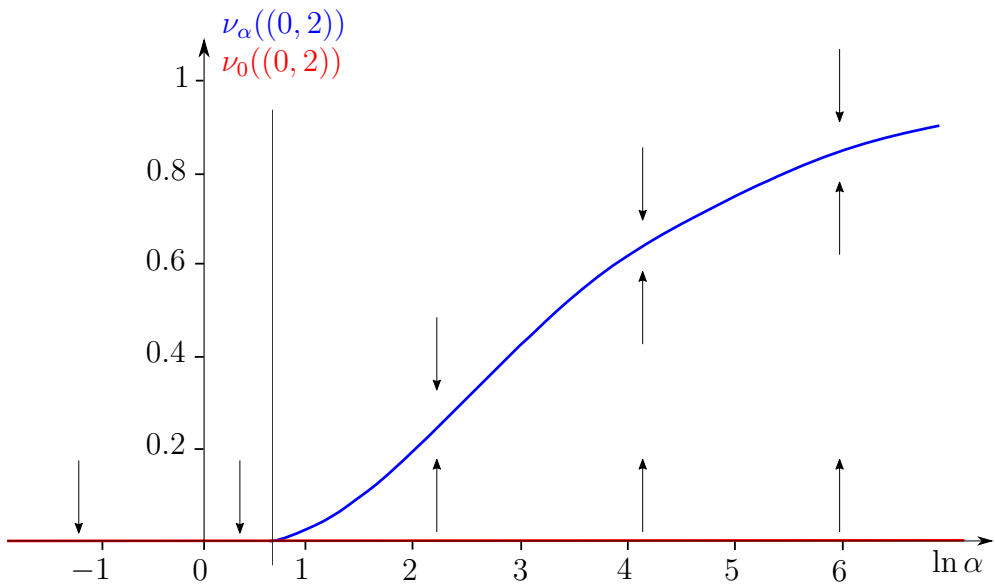}
    \caption{\label{fig:simulations} Domain of attraction of the attractors, with quadratic $\varphi^1,\varphi^2,\psi^1,\psi^2$. $\nu_\alpha((0,2))$ has been computed numerically using Fleming-Viot type approximation techniques (see for instance~\cite{DelMoralVillemonais2018}).}
\end{figure}

Our goal in this note is to prove a similar property for the family $(X^\alpha)_{\alpha \geq 0}$, which can be formulated as a
bifurcation of the dynamical system on $\mathcal{P}(D)$ generated by $\Phi^\alpha$. Our motivation is to quantify precisely the speed of convergence and
the basin of attraction of the fixed points of this dynamical systems (see Fig.~\ref{fig:simulations}). Another motivation is to provide
an illustration to the fact that, for an absorbed diffusion process satisfying the weak H\"ormander condition and regularity
properties at the absorbing set, uniqueness of a quasi-stationary distribution does not necessary hold true unless some accessibility
properties are satisfied (see Theorem~1.8 of~\cite{BenaimChampagnatEtAl2021}). The definition of a quasi-stationary distribution is
recalled in Section~\ref{sec:abstractQSD} below.

To state our main result, we define the absorption parameters of $X^\alpha$ for $\alpha=1$:
\begin{align*}
\lambda_1:=\inf\left\{\lambda\in\mathbb R,\ \liminf_{t\to+\infty} e^{\lambda t}\,\mathbb P_{x_1}(X^1_t\in(0,2))>0\right\},\ \text{ for some $x_1\in (0,2))$}.
\end{align*}
and
\begin{align*}
\lambda_2:=\inf\left\{\lambda\in\mathbb R,\ \liminf_{t\to+\infty} e^{\lambda t}\,\mathbb P_{x_2}(X^1_t\in[3,5))>0\right\},\ \text{ for some $x_2\in [3,5)$},
\end{align*}
As expected, we will see that the parameters $\lambda_1$ and $\lambda_2$ are positive and do not depend on $x_1$ nor 
$x_2$.

%\denIL{Dans le théorème, je suis parti de la formulation utilisé par Michel p.9 dans [Stochastic approximation of quasi-stationarydistributions on compact spaces and applications]. Comme on s'éloigne très nettement de mon domaine de compétence, il faudra discuter d'une éventuelle meilleure formulation.}
\begin{theorem}
	\label{thm:main} The dynamical system generated by $\Phi^\alpha$ parametrized by $\alpha>0$ admits a transcritical bifurcation at $\lambda_1/\lambda_2$. More precisely, there exist a family of probability measure $(\mu_\alpha)_{\alpha>\nicefrac{\lambda_1}{\lambda_2}}$ on $(0,5)$ and a probability measure $\mu_0$ on $[3,5)$ such that: 
	\begin{itemize}
        \item for $\alpha\leq \lambda_1/\lambda_2$, $\mu_0$ is a global attractor for $\Phi^\alpha$ for the total
            variation distance,
        \item for $\alpha>\lambda_1/\lambda_2$, $\mu_0$ is a saddle point whose stable manifold for the total
            variation distance is the set $W^s(\mu_0):=\{\mu\in \mathcal P((0,5)):\mu((0,2))=0\}$, and $\mu_\alpha$ is a stable point
          whose basin of attraction for the total variation distance is $W^u(\mu_0)=\{\mu\in \mathcal P((0,5)): \mu((0,2))>0\}$.
	\end{itemize}
\end{theorem}

    Our method also provides an estimate for the speed of convergence of the dynamical system generated by
    $\Phi^\alpha$ to its limit fixed point. One can check from the proof that it is exponential when $\alpha\neq \lambda_1/\lambda_2$
    and polynomial in $O(1/t)$ when $\alpha=\lambda_1/\lambda_2$.

%\denIL{TODO? graphe montrant l'\'evolution de l'ensemble des points stationnaires gr\^ace \`a des simulations num\'eriques.}
 
 The proof of this result relies on the theory of quasi-stationary distributions. We show in particular that the process $X^\alpha$
 admits either one or two quasi-stationary distributions, depending on the value of $\alpha$, which correspond to $\mu_0$ and
 $\mu_\alpha$ (when $\alpha>\lambda_0/\lambda_1$). A central feature allowing this property is that $X^\alpha$ is reducible. Indeed,
 it is proved in~\cite{BenaimChampagnatEtAl2021} that irreducibility for such diffusions entails the uniqueness of a quasi-stationary
 distribution. In order to prove Theorem~\ref{thm:main}, we start with considerations on quasi-stationary distributions for Markov
 processes in reducible state spaces. Since they apply to general Markov processes and may have independent interest, they are stated
 independently in Section~\ref{sec:abstractQSD}. We conclude the proof of the theorem in Section~\ref{sec:proofthm}.

\section{Quasi-stationarity for two-clusters reducible processes}
\label{sec:abstractQSD}
Let $X$ be a Markov process with state space $M=D\cup \{\partial\}$ with $\partial\notin D$, in discrete or continuous time, such
that $M$ admits a measurable partition $D_1\cup D_2\cup\{\partial\}$. We assume that $\{\partial\}$ and $D_2\cup \{\partial\}$ are
absorbing sets (see Figure~\ref{fig:trangraph}), which means that
\begin{align*}
\mathbb P_\partial (X_t=\partial)=1\text{ and }
\mathbb P_x(X_t\in D_2\cup\{\partial\})=1,\ \forall x\in D_2,\ \forall t\geq 0.
\end{align*}

\begin{figure}[h]
	\center
	\includegraphics[width=6cm]{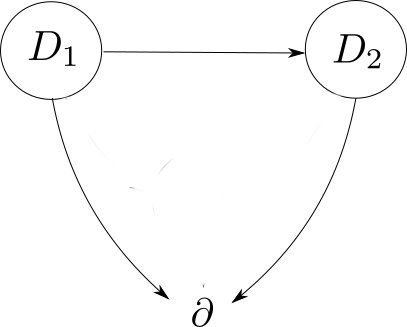}
	\caption{Transition graph displaying the relation between the sets $D_1$, $D_2$ and $\d$.}
	\label{fig:trangraph}
\end{figure}

A probability measure $\nu$ is said to be a \textit{quasi-stationary distribution} if, for all $t \geq 0$,
\begin{equation}\label{qsd}\P_\nu(X_t \in \cdot | X_t \ne \d) = \nu(\cdot).\end{equation}
It is well-known (see \cite[Proposition 1]{MeleardVillemonais2012}) that the notion of quasi-stationary distribution is equivalent to the one of \textit{quasi-limiting distribution}, defined as a probability measure $\nu$ such that there exist some initial distributions $\mu$ such that, for all measurable subset $A \subset D$, 
$$\lim_{t \to \infty} \P_\mu(X_t \in A | X_t \ne \d) = \nu(A).$$
It is also well-known that, to any quasi-stationary distribution $\nu$ is associated the so-called \emph{exponential
    absorption rate} $\lambda_0>0$ such that $\mathbb{P}_\nu(X_t\in\cdot,\ X_t\neq\d)=e^{-\lambda_0 t}\nu$.
We refer the reader to \cite{MeleardVillemonais2012,ColletMartinezEtAl2013,DoornPollett2013} for a general overview on the theory of quasi-stationarity,  and to~\cite{Pinsky1985,GongQianEtAl1988,CattiauxColletEtAl2009,CattiauxMeleard2010,KnoblochPartzsch2010,KolbSteinsaltz2012,LittinC.2012,DelMoralVillemonais2018,ChampagnatVillemonais2017a,ChampagnatCoulibaly-PasquierEtAl2018,HeningKolb2019,GuillinNectouxEtAl2020,LelievreRamilEtAl2021,Ramil2021} for the study of the quasi-stationary distribution of diffusion processes.

The aim of this section is to provide conditions on $X$ allowing to obtain the existence of a quasi-stationary distribution $\nu$ for the process $X$, as well as the so-called \textit{Malthusian behavior} (see \cite{BertoinWatson2020} for the terminology), that is to say the existence of a positive function $\eta$ on $D$ such that
$$\lim_{t \to \infty} e^{\lambda_0 t} \P_x(X_t \in \cdot, X_t \ne \d) = \eta(x) \nu(\cdot),$$
where $\lambda_0$ is the exponential absorption rate associated to $\nu$ (in particular, this convergence entails the convergence of
the conditional probability measure $\P_x(X_t \in \cdot | X_t \ne \d)$ towards $\nu$). In what follows, we will present three
different sets of assumptions, each discussed in three different subsections, and each entailing different Malthusian behavior for
the process $X$.

\subsection{Exponential convergence on $D_2$ and faster exit from $D_1$}
\label{assumption-QSD2}

Let us introduce our first set of assumptions. \medskip

\medskip\noindent\textbf{Assumption QSD2.} \begin{enumerate}[a.]\item There exist a positive function $\eta_2$ on $D_2$, a probability measure $\nu_2$ on $D_2$, a constant $\lambda_2 > 0$, and positive constants $C_2,\gamma_2>0$ such that
\begin{align}
\label{eq:eta2def}
\left\|e^{\lambda_2 t}\mathbb P_x(X_t\in \cdot, X_t \in D_2)-\eta_2(x)\nu_2(\cdot)\right\|_{TV}\leq C_2e^{-\gamma_2 t},\ \forall x\in D_2,\ t\geq 0.
\end{align}
\item In addition, $\sup_{x\in D_1}e^{\lambda_2 t}\mathbb P_x(X_t\in D_1)\leq f(t)$, where $f$ is non-increasing and $L^1$ on
  $\mathbb{R}_+$.
\end{enumerate}

The assumption a.\ refers to the Malthusian behavior, as described before, of the restriction of $X$ on the subset $D_2$, holding
uniformly in $x$ in total variation and exponentially fast. We refer the reader to \cite{DelMoralMiclo2002,DelMoral2004,ChampagnatVillemonais2017b,FerreRoussetEtAl2018,Velleret2018,BansayeCloezEtAl2019,ChampagnatVillemonais2019,GuillinNectouxEtAl2020,BenaimChampagnatEtAl2021,ChampagnatVillemonais2021}
 for general criteria entailing such behavior. % The assumption b. implies that $\sup_{x \in D_1} \P_x(X_t \in D_1)$ vanishes faster than $e^{-\lambda_2 t}$. 

Also, remark that the inequality \eqref{eq:eta2def} entails that $\eta_2$ is bounded
   (take for example $t=0$). In addition, it implies that $\eta_2(x)=\lim_{t\rightarrow+\infty} e^{\lambda_2
    t}\mathbb{P}_x(X_t\neq\d)$ for all $x\in D_2$. In addition, as noticed above, it implies that $\nu_2$ is a quasi-stationary
  distribution for $X$.

Then Assumption QSD2 entails the following result on quasi-stationarity. 

\begin{theorem}
\label{thm-qsd2}
Under QSD2% .a., $\nu_2$ is a quasi-stationary distribution for $X$.  If moreover Assumption QSD2.b. holds true
, there exists $\eta:D\to\mathbb R_+$ positive on $D_2$ such that
        \begin{align}
        \label{eq:thmcase1}
        \sup_{x\in D}\left\|e^{\lambda_2 t}\mathbb P_x(X_t\in \cdot, X_t \neq\d)-\eta(x)\nu_2(\cdot\cap D_2)\right\|_{TV}
       \xrightarrow[t\to+\infty]{}0.
        \end{align}
	 If in addition $\mathbb P_x(\exists n\geq 0,\,X_n\in D_2)>0$ for all $x\in D_1$, then $\eta$ is positive on $D$ and $\nu_2(\cdot\cap D_2)$ is the unique quasi-stationary distribution for $X$ on $D$.
\end{theorem}

In other terms, the Malthusian behavior of $X$, only assumed for $x \in D_2$ in Assumption QSD2, holds for all $x \in D$ and uniformly in $x$ in total variation. Speed of convergence for \eqref{eq:thmcase1} is discussed after the proof of Theorem \ref{thm-qsd2}.

\begin{proof}[Proof of Theorem \ref{thm-qsd2}.]
% The first sentence of the theorem follows from the explanation did in the paragraph above Section \ref{assumption-QSD2}.
For the rest, we first prove this theorem in the discrete time setting, and then consider the continuous-time setting.
\medskip \\
\textbf{Step 1: Proof in the discrete-time setting.}
    We define the stopping time $\tau_1^c:=\min\{n\geq 0,\ X_n\notin D_1\}$.
    For all  $x\in D$ and all measurable set $A\subset D_1\cup D_2$, for all $n\in\mathbb Z_+$, we have,  using the strong Markov property at time $\tau_1^c$,
    \begin{align}
    \mathbb P_x(X_n\in A)&=\mathbb P_x(X_n\in A\cap D_1)
    +\mathbb P_x(X_n\in A\cap D_2)\notag\\
    &=\mathbb P_x(X_n\in A\cap D_1)
    +\sum_{k=0}^{n} \mathbb E_x\left(1_{\tau_1^c=k}\,\mathbb P_{X_k}(X_{n-k}\in A\cap D_2)\right)\label{eq:decomp}.
    \end{align}
    
    If $x\in D_2$, then the result is an immediate consequence of~\eqref{eq:eta2def} with $\eta(x)=\eta_2(x)$. It only remains to consider the case $x\in D_1$.
    On the one hand, we have by assumption
    \begin{align}
    e^{\lambda_2 n}\mathbb P_x(X_n\in A\cap D_1)&\leq e^{\lambda_2 n} \mathbb P_x(X_n\in D_1)\xrightarrow[n\to+\infty]{} 0  .\label{eq:useME1step1}
    \end{align}
    On the other hand, using~\eqref{eq:eta2def} and extending $\eta_2$ to $\{\partial\}$ by $\eta_2(\partial)=0$, we obtain, for all $k\geq 0$, $t\geq 0$ and measurable set $A\subset D_2$,
    \begin{align*}
    & \Big|\mathbb E_x\left(1_{\tau_1^c=k}\,\mathbb P_{X_k}(X_{n-k}\in A)\right) -\mathbb E_x\left(1_{\tau_1^c=k}\,\eta_2(X_k)\nu_2(A) e^{-\lambda_2 (n-k)} \right)\Big|\\
    &\phantom{\mathbb E_\mu1_{\tau_1^c=k}\,\mathbb P_{X_k}(X_{n-k}}\leq \mathbb E_x\left(1_{\tau_1^c=k}\,C_2\, e^{-(\lambda_2+\gamma_2) (n-k)} 1_{X_k\in D_2}\right)\\
    &\phantom{\mathbb E_\mu1_{\tau_1^c=k}\,\mathbb P_{X_k}(X_{n-k}}\leq  
    \begin{cases}
    C_2\,\mathbb P_x\left(X_{k-1}\in D_1,\,X_k\in D_2\right)\, e^{-(\lambda_2+\gamma_2) (n-k)}&\text{ if }k\geq 1,\\
    0&\text{ if }k=0.
    \end{cases}
    \end{align*}
    Summing over $k$, we deduce that, for all $n\geq 0$ and all measurable set $A\subset D_2$,
    \begin{align}
    &\left|\sum_{k=0}^{n} \mathbb E_x\left(1_{\tau_1^c=k}\,\mathbb P_{X_k}(X_{n-k}\in A)\right)-\sum_{k=0}^{n} \mathbb E_x\left(1_{\tau_1^c=k}\eta_2(X_k)\nu_2(A)e^{-\lambda_2 (n-k)} \right)\right|\notag\\
    &\leq C_2e^{-\lambda_2 n}e^{-\gamma_2 n}\sum_{k=1}^n  e^{\lambda_2 k}e^{\gamma_2 k} \mathbb P_x\left(X_{k-1}\in D_1,\,X_k\in
      D_2\right) \notag\\ & \leq C_2e^{-\lambda_2 (n-1)}e^{-\gamma_2 n}\sum_{k=1}^n  e^{\gamma_2 k} f(k-1)\label{eq:gammabound}.
    \end{align}
   Moreover
    \begin{align*}
    \sum_{k\geq n+1} \mathbb E_x\left(1_{\tau_1^c=k}\eta_2(X_k)e^{\lambda_2 k} \right)
    &\leq \|\eta_2\|_\infty \sum_{k\geq n+1} e^{\lambda_2 k}\mathbb P_x(X_{k-1}\in D_1,\,X_{k}\in D_2) \\ & \leq \|\eta_2\|_\infty
    \sum_{k\geq n+1} e^{\lambda_2 k}\left(\mathbb P_x(X_{k-1}\in D_1)-\mathbb{P}_x(X_{k}\in D_1)\right) \\ & \leq \|\eta_2\|_\infty
    (e^{\lambda_2}-1)\sum_{k\geq n} e^{\lambda_2 k}\mathbb{P}_x(X_{k}\in D_1) \\ & \leq \|\eta_2\|_\infty
    (e^{\lambda_2}-1)\int_{n-1}^{+\infty}f(t)dt.
    \end{align*}

    This,~\eqref{eq:decomp},~\eqref{eq:useME1step1} and~\eqref{eq:gammabound} % and \eqref{uniform-finite-sum}
    entail that, for all $x\in D_1$ and all measurable $A\subset D$,
    \begin{align*}
    \sup_{x\in D_1} \left|e^{\lambda_2 n}\mathbb P_x(X_n\in A)-\sum_{k=0}^{+\infty} \mathbb E_x\left(1_{\tau_1^c=k}\eta_2(X_k)e^{\lambda_2 k} \right)\nu_2(A\cap D_2)\right|
    \xrightarrow[n\to+\infty]{} 0.
    \end{align*}
    Setting
      \[
       \eta(x):=\begin{cases}
       \sum_{k=0}^{+\infty} \mathbb E_x\left(1_{\tau_1^c=k}\eta_2(X_k)e^{\lambda_2 k} \right)&\text{ if }x\in D_1\\
        \eta_2(x)&\text{ if }x\in D_2,
       \end{cases}
      \]
%      and, for all $x\in D_1$,
%            \[
%            \varepsilon_n(x)=
%            C_2 \sum_{k=1}^n e^{\gamma_2 (k-n)}\delta_x P^{k-1}Q1+ (1+\|\eta_2\|_\infty) \sum_{k= n+1}^\infty e^{\lambda_2 k}\delta_x P^{k-1}Q1
%%            C_2 e^{-\gamma_2 n}&\text{ if }x\in D_2.
%             \]
%             and $\varepsilon_n(x)=C_2e^{-\gamma_2 n}$ for all $x\in D_2$,
 this concludes the proof of~\eqref{eq:thmcase1}.

 If in addition $\mathbb P_x(\exists n\geq 0,\ X_n\in D_2)>0$ for all $x \in D_1$, then $\eta$ is positive on $D$ and hence $\nu_2(\cdot\cap D_2)$ is the unique quasi-limiting distribution of the process and hence its unique quasi-stationary distribution.
\medskip \\
\textbf{Step 2 : Proof in the continuous-time setting.} The proof is done by applying the discrete time result to the Markov chain $(X_n)_{n\in\mathbb Z_+}$.

	Let $X$ satisfy Assumption~QSD2. Then the discrete time process $(X_n)_{n\in\mathbb N}$ also satisfies Assumption~QSD2 and hence, by Theorem~\ref{thm-qsd2} in the discrete time setting, we have
	\begin{align}
	\sup_{x\in D}\left\|e^{\lambda_2 n}\mathbb P_x(X_n\in \cdot,\,X_n\neq \partial)-\eta(x)\nu_2(\cdot\cap D_2)\right\|_{TV}
	\xrightarrow[n\to+\infty]{}0. \label{eq:pf-QSD2}
	\end{align}
	For any $x\in D$ and $h>0$, integrating the above convergence result with respect to $\mathbb P_x(X_h\in\cdot)$ entails that
	\begin{align*}
	\left\|e^{\lambda_2 n}\mathbb P_x(X_{n+h}\in\cdot, {X_{n+h}\neq \partial})-\mathbb E_x(\eta(X_h))\nu_2(\cdot\cap D_2)\right\|_{TV}
	\xrightarrow[n\to+\infty]{}0
	\end{align*}
	Similarly, for any $x\in D$ and $h>0$, applying~\eqref{eq:pf-QSD2} to the test function $y\in D\mapsto \mathbb P_y(X_h\in\cdot,\,X_h\neq \partial)$, gives
	\begin{align*}
	\left\|e^{\lambda_2 n}\mathbb P_x(X_{n+h}\in\cdot, {X_{n+h}\neq\partial})-\eta(x)\mathbb P_{\nu_2}(X_h\in\cdot,\,X_h\neq\partial)\right\|_{TV}
	\xrightarrow[n\to+\infty]{}0.
	\end{align*}
       But $D_2\cup\{\partial\}$ being absorbing, we have $\mathbb P_{\nu_2}(X_h\in\cdot,\,X_h\neq\partial)=\mathbb P_{\nu_2}(X_h\in\cdot,\,X_h\in D_2)$, which implies that
	\begin{align*}
	\mathbb E_x(\eta(X_h))\nu_2(\cdot\cap D_2)=\eta(x)\mathbb P_{\nu_2}(X_h\in\cdot\cap D_2)=e^{-\lambda_2 h} \eta(x)\nu_2(\cdot\cap D_2),
	\end{align*}
	so that $\mathbb E_x(\eta(X_h))=e^{-\lambda_2 h} \eta(x)$, and hence
	\begin{align*}
	\left\|e^{\lambda_2 n}\mathbb P_x(X_{n+h}\in\cdot, {X_{n+h}\neq\partial})-e^{-\lambda_2 h} \eta(x)\nu_2(\cdot\cap D_2)\right\|_{TV}
	\xrightarrow[n\to+\infty]{}0
	\end{align*}
	Since the convergence holds uniformly in $h\in[0,1]$, this concludes the proof of~\eqref{eq:thmcase1} in the continuous time setting. The uniqueness of the quasi-stationary distribution is immediate, since any quasi-stationary distribution for $(X_t)_{t\in\mathbb R_+}$ is also a quasi-stationary distribution for $(X_n)_{n\in\mathbb Z_+}$.
\end{proof}

Let us do some remarks before passing to the second set of assumptions.

\begin{remark}
\label{remark1}
	In the proof, the speed of convergence in the above theorem are explicit in terms of the quantities appearing in the assumptions. In particular, if Assumption~QSD2 holds true with 
	\[
	e^{(\lambda_2+\varepsilon) t}\sup_{x\in D_1}\mathbb P_x\left(X_t\in D_1\right)\xrightarrow[t\to+\infty]{} 0,
	\]
	 for some $\varepsilon>0$, then the convergence in~\eqref{eq:thmcase1} is exponential.
\end{remark}

\begin{remark}
\label{remark2}
	Non-uniform speed of convergence can also be proved under weaker form of Assumption~QSD2. For instance
        if~\eqref{eq:eta2def} holds true and if, for some $x\in D_1$, $e^{\lambda_2 t}\mathbb P_x(X_t\in D_1)\leq f(t)$ with $f$
        non-increasing and $L^1$ on $\mathbb{R}_+$
%         converges to $0$ when $t\to+\infty$ and
% \[
% \sum_{k\geq 1} e^{\lambda_2 k} \mathbb P_x (X_{k-1}\in D_1,\, X_k \not \in D_2)<+\infty.
% \]
(but non-uniformly in $x\in D_1$), then,
	\[
	\left\|e^{\lambda_2 n}\mathbb P_x(X_n\in \cdot)-\eta(x)\nu_2(\cdot\cap D_2)\right\|_{TV}
	\xrightarrow[n\to+\infty]{}0.
	\]
\end{remark}

\begin{remark}
\label{remark3}
	The use of the absorbed Markov process setting is for convenience only: the above result applies more generally to semi-groups on $L_\infty(D)$ admitting an isolated simple leading eigenvalue $\lambda\in\mathbb C$.
\end{remark}

\subsection{Exponential convergence on $D_1$ and faster absorbtion in $D_2$}

Let us now state our second set of assumptions.

\medskip\noindent\textbf{Assumption QSD1.} There exist a positive function $\eta_1$ on $D_1$, a probability measure $\nu_1$ on $D_1$, a constant $\lambda_1 > 0$, and positive constants $C_1,\gamma_1>0$ such that
\begin{align}
   \label{eq:eta1def}
   \left\|e^{\lambda_1 t}\mathbb P_x(X_t\in \cdot, X_t \in D_1)-\eta_1(x)\nu_1(\cdot)\right\|_{TV}\leq C_1e^{-\gamma_1 t},\ \forall x\in D_1,\ t\geq 0.
\end{align}
In addition, for all $t\geq 0$,
\begin{equation}
    \label{finite-sum-QSD1}
    % \sum_{k=0}^\infty e^{\lambda_1 k} \sup_{x \in D_2} \P_x(X_k \in D_2) < + \infty.
    e^{\lambda_1 t} \sup_{x \in D_2} \P_x(X_t \in D_2)\leq f(t)
\end{equation}
where $f$ is non-decreasing and $L^1$ on $\mathbb{R}_+$.

Assumption QSD1 is very similar to Assumption QSD2, except that this assumption now deals with the quasi-stationarity for the
restriction of the process $X$ considered as absorbed by $D_2 \cup \{\d\}$. Similarly to Assumption QSD2, Assumption QSD1 entails
that $\eta_1$ is bounded, $\eta_1(x) = \lim_{t \to \infty} e^{\lambda_1 t} \P_x(X_t \in D_1)$ for all $x \in D_1$, and that $\nu_1$
is a quasi-stationary distribution for the process $X$ considered as absorbed by $D_2 \cup \{\d\}$. The condition
\eqref{finite-sum-QSD1} tells that $\sup_{x \in D_2} \P_x(X_t \ne \d)$ vanishes faster than the probability of survival starting from
the quasi-stationary distribution $\nu_1$.

% Now, let us introduce the result entailed by Assumption QSD1.

\begin{theorem}
\label{thm-qsd1}
 Under Assumption~QSD1, there exists a positive finite measure $\nu$ on $D$ such that
        \begin{align}
        \sup_{x\in D} \left\|e^{\lambda_1 t}\mathbb P_x(X_t\in  \cdot, X_t \neq\d)-\eta(x)\nu \right\|_{TV}
        \xrightarrow[t\to+\infty]{} 0\label{eq:thmcase2},
        \end{align}
where $\eta(x)=\eta_1(x)$ for all $x\in D_1$ and $\eta(x)=0$ for all $x\in D_2$.
		In addition, $\nu/\nu(D)$ is the unique quasi-stationary distribution of $X$ such that $\nu(D_1)/\nu(D)>0$.
\end{theorem}

Hence, \eqref{eq:thmcase2} entails that $\P_x(X_t \in \cdot | X_t \ne \d)$ converges in total variation towards $\nu/\nu(D)$. Note also that Assumption QSD1 does not tell anything on the convergence of $\P_x(X_t \in \cdot | X_t \ne \d)$ when $x \in D_2$.  

\begin{remark}
\label{remark4}
Similarly as in Remark~\ref{remark1}, if 
\[
e^{(\lambda_1+\epsilon)t}\sup_{x\in D_2} P_x(X_k\in D_2)\to 0
\]
for some $\varepsilon>0$,
	 then the convergence in~\eqref{eq:thmcase2} is exponential.
\end{remark}

\begin{remark}
\label{remark5}
Similarly as in Remark~\ref{remark2}, 
	 if~\eqref{eq:eta1def} holds true and if 
\begin{equation*}
    \sum_{k=0}^\infty e^{\lambda_1 k} \sup_{x \in D_2} \P_x(X_k \in A) < + \infty.
\end{equation*}
 for some measurable $A\subset D$, then 
	\[
	\sup_{x\in D_1} \left|e^{\lambda_1 n}\mathbb P_x(X_n\in A)-\eta_1(x)\nu(A) \right|
	\xrightarrow[n\to+\infty]{} 0.
	\]
	This last case is particularly interesting, since it applies to situations where $\nu$ is not necessarily a finite measure (one may have $\nu(D_2)=+\infty$).
\end{remark}

\begin{proof}[Proof of Theorem \ref{thm-qsd1}.] We only prove the result in the discrete-time setting. The adaptation to the
  continuous-time setting follows from the same argument as in the proof of Theorem \ref{thm-qsd2}. 
 
    Fix $x\in D_1$.  The result is an immediate consequence of~\eqref{eq:eta1def} if $A\subset D_1$ with
      $\nu(\cdot\cap D_1)=\nu_1$. It remains to consider the case $A\subset D_2$, the general case being obtained by linearity.
    For all measurable $A\subset D_2$ and all $x\in D_1$, we have, for all $k\in\{0,\ldots,n-1\}$,
    \begin{align*}
    &\left|\mathbb E_x\left(1_{\tau_1^c=n-k}\mathbb P_{X_{n-k}}(X_{k}\in  A) \right)-\eta_1(x)e^{-\lambda_1 (n-k-1)}\mathbb E_{\nu_1}\left(1_{\tau_1^c=1}\,1_{X_{k+1}\in  A} \right)\right|\\
    &=\left|\mathbb E_x\left(f_A(X_{n-k-1})1_{n-k-1<\tau_1^c} \right)-\eta_1(x)e^{-\lambda_1 (n-k-1)} \nu_1(f_A) \right|
    \end{align*}
    where, for all $y\in D_1$, $f_A(y)=\mathbb P_y(X_1 \not \in D_1,\ X_{k+1}\in A)$, and where $\tau_1^c$ was defined in the proof of Theorem \ref{thm-qsd2}. By Markov's property, we have
    $f_A(y)\leq \sup_{z\in D_2} \mathbb P_z(X_{k}\in A)$.
    Hence, according to~\eqref{eq:eta1def}, we have
    \begin{align*}
     &\left|\mathbb E_x\left(1_{\tau_1^c=n-k}\mathbb P_{X_{n-k}}(X_{k}\in  A) \right)-\eta_1(x)e^{-\lambda_1 (n-k-1)}\mathbb E_{\nu_1}\left(1_{\tau_1^c=1}\,1_{X_{k+1}\in  A} \right)\right|\\
    &\leq C_1 e^{-\lambda_1 n}e^{-\gamma_1 (n-k-1)} \left(e^{\lambda_1 (k+1)}\sup_{z\in D_2} \mathbb P_z(X_{k}\in A)\right) \\ & \leq C_1 e^{-\lambda_1 n}e^{-\gamma_1 (n-k-1)}e^{\lambda_1}f(k).
    \end{align*}
    Summing over $k\in\{0,\ldots,n-1\}$ and multiplying by $e^{\lambda_1 n}$, we get
    \begin{align*}
    &\left|e^{\lambda_1 n}\mathbb P_x\left(X_n\in  A \right)-\eta_1(x)\sum_{k=0}^{n-1} e^{\lambda_1 (k+1)}\mathbb E_{\nu_1}\left(1_{\tau_1^c=1}\,1_{X_{k+1}\in  A} \right)\right|\\
&\leq C_1  e^{-\gamma_1 n} \sum_{k=0}^{n-1} e^{\gamma_1 (k+1)} e^{\lambda_1(k+1)}\sup_{z\in D_2} \mathbb P_z(X_{k}\in A) \\ & \leq C_1  e^{\lambda_1}\sum_{k=0}^{n-1}e^{-\gamma_1 (n-k-1)}f(k),
    \end{align*}
    where the right hand term goes to $0$ when $n\to+\infty$. % by \eqref{finite-sum-QSD1}.
	Finally, we observe that
	\begin{align*}
	\sum_{k=n}^{+\infty} e^{\lambda_1 (k+1)}\mathbb E_{\nu_1}\left(1_{\tau_1^c=1}\,1_{X_{k+1}\in  A} \right) & \leq
        \sum_{k=n}^{+\infty} e^{\lambda_1 (k+1)}\sup_{z\in D_2}\mathbb P_x(X_{k}\in  A) \\ & \leq e^{\lambda_1}\int_{n-1}^{+\infty}f(t)dt,
	\end{align*}
	which also goes to $0$ when $n\to+\infty$ under the assumption \eqref{finite-sum-QSD1}. This concludes the proof of~\eqref{eq:thmcase2} with
	\begin{align*}
	\nu(A):=\nu_2(A\cap D_1)+\sum_{k=0}^{+ \infty} e^{\lambda_1 (k+1)}\mathbb E_{\nu_1}\left(1_{\tau_1^c=1}\,1_{X_{k+1}\in  A\cap D_2} \right)
	\end{align*}
	In particular, $\nu/\nu(D)$ is a quasi-limiting distribution of $X$ and is thus a quasi-stationary distribution.
	
	To conclude, let $\nu'$ be a quasi-stationary distribution for $X$ such that $\nu'(D_1)>0$. Integrating~\eqref{eq:thmcase2}
        with respect to $\nu'$ and noting that $\nu'(\eta_1)>0$, we deduce that the exponential absorption rate of
          $\nu'$ is $\lambda_1$ and that $\nu/\nu(D)$ is a quasi-limiting distribution for $X$ starting from $\nu'$, and thus $\nu=\nu'$.
\end{proof}

\subsection{Exponential convergence in $D_1$ and $D_2$ with the same rate}

Let us now present our last set of assumptions.

\medskip\noindent\textbf{Assumption QSD1-2.} There exist two positive functions $\eta_1$ on $D_1$ and $\eta_2$ on $D_2$, two probability measures $\nu_1$ on $D_1$ and $\nu_2$ on $D_2$, a positive constant $\lambda_0 > 0$, and positive constants $C_1,\gamma_1,C_2,\gamma_2>0$ such that~\eqref{eq:eta1def} and~\eqref{eq:eta2def} hold true with $\lambda_1=\lambda_2=\lambda_0$. \medskip

In other terms, $\nu_2$ and $\nu_1$ are the quasi-stationary distributions for $X$ respectively started from $D_2$ and absorbed at
$\d$, and started from $D_1$ and absorbed at $D_2 \cup \{\d\}$, associated to the same absorption rate $\lambda_0 > 0$. Under this assumption, we have the following result.

\begin{theorem}
\label{thm-qsd1-2}
Under Assumption~QSD1-2, the process admits $\nu_2(\cdot\cap D_2)$ as unique quasi-stationary distribution and
        \begin{align}
        \label{eq:thmcase3}
        \sup_{x\in D_1}\left\| \frac{e^{\lambda_0 t}}{t}\mathbb P_x\left({X_t\in \cdot}, X_t \in D_2\right) 
        -\eta(x)\nu_2(\cdot\cap D_2)\right\|_{TV}\leq \frac{C}{t+1},
        \end{align}
        where $\eta$ is a positive function on $D_1$ and $C$ is a positive constant. 
\end{theorem}

In particular, Malthusian behavior \eqref{eq:eta2def} does not hold for $x \in D_1$. However, \eqref{eq:thmcase3} still entails that
the probability measure $\P_x(X_t \in \cdot | X_t \ne \d)$ converges in total variation towards $\nu_2$ for all $x \in D_1$ (and also for all $x \in D_2$ by Hypothesis~\eqref{eq:eta2def}). 

\medskip

	   \begin{proof}[Proof of Theorem \ref{thm-qsd1-2}.] As for the proof of Theorem \ref{thm-qsd2}, we only deal with the discrete-time setting.  

 %    If $x\in D_1$ and $A\subset D_1$, or if $x\in D_2$, the result is an immediate consequence of~\eqref{eq:eta1def}
 %    and~\eqref{eq:eta2def}. 
Fix $x\in D_1$ and
% a
measurable set $A\subset D_2$. For all $k\geq 0$, we have, using~\eqref{eq:eta2def} (recall that $\lambda_1=\lambda_2=\lambda_0$),
\begin{align}
\label{eq:step31}
&\left|e^{\lambda_0 (k+1)}\mathbb P_{\nu_1}(\tau_1^c=1,\,X_{k+1}\in A)
-e^{\lambda_0}\mathbb E_{\nu_1}\left(1_{\tau_1^c=1}\eta_2(X_1)\right)\nu_2(A)\right|\leq e^{\lambda_0}C_2e^{-\gamma_2 k},
\end{align}
where $\tau_1^c$ was defined in the proof of Theorem \ref{thm-qsd2}. 
Moreover, for all $n\geq 1$ and $k\in\{0,\ldots,n-1\}$,
\begin{align*}
&\left|e^{\lambda_0 (k+1)}\eta_1(x)\mathbb P_{\nu_1}(\tau_1^c=1,\,X_{k+1}\in A)
-e^{\lambda_0 n}\mathbb E_x\left(1_{\tau_1^c=n-k}\,1_{X_n\in A}\right)\right|\\
&= e^{\lambda_0 (k+1)}\left|\eta_1(x)\nu_1(f_A)
-e^{\lambda_0 (n-k-1)}\mathbb E_x\left(f_A(X_{n-k-1})\right)\right|,
\end{align*}
where $f_A(y)=\mathbb P_y(\tau_1^c=1,\,X_{k+1}\in A)\leq \sup_{z\in D_2} \mathbb P_z(X_k\in D_2)$. Hence using~\eqref{eq:eta1def}, we deduce that
\begin{align*}
&\left|e^{\lambda_0 (k+1)}\eta_1(x)\mathbb P_{\nu_1}(\tau_c^1=1,\,X_{k+1}\in A)
-e^{\lambda_0 n}\mathbb E_x\left(1_{\tau_1^c=n-k}\,1_{X_n\in A}\right)\right|\\
&\leq e^{\lambda_0 (k+1)}C_1 e^{-\gamma_1(n-k-1)} \sup_{z\in D_2} \mathbb P_z(X_k\in D_2)\\
&=C_1 e^{\lambda_0}  e^{-\gamma_1(n-k-1)}  \sup_{z\in D_2} e^{\lambda_0 k}\mathbb P_z(X_k\in D_2),
\end{align*}
where, according to~\eqref{eq:eta2def}, $ \sup_{z\in D_2} e^{\lambda_0 k}\mathbb P_z(X_k\in D_2)$ is uniformly bounded in $k$ by $C_2+\|\eta_2\|_\infty$. This,~\eqref{eq:step31} and summing over $k\in\{0,\ldots,n-1\}$ imply that
\begin{align*}
&\left| e^{\lambda_0 n}\mathbb P_x\left({X_n\in A}\right) 
-ne^{\lambda_0}\eta_1(x)\mathbb E_{\nu_1}\left(1_{\tau_1^c=1}\eta_2(X_1)\right)\nu_2(A)\right|\\
&\leq C_2 \eta_1(x) e^{\lambda_0}\sum_{k=0}^\infty e^{-\gamma_2 k}+\sum_{k=0}^{+\infty} e^{-\gamma_1 k}C_1e^{\lambda_0}(C_2+\|\eta_2\|_\infty).
\end{align*}
This concludes the proof of~\eqref{eq:thmcase3} with $\eta(x) := e^{\lambda_0} \eta_1(x) \mathbb{E}_{\nu_1}\left(\mathbbm{1}_{\tau_1^c = 1} \eta_2(X_1)\right)$.
\end{proof}

\section{Proof of Theorem~\ref{thm:main}}
\label{sec:proofthm}

We start with a proposition related to the theory of quasi-stationary distributions for diffusion processes.

\begin{proposition}
	\label{prop:QSDs}
	There exist a positive function $\eta_2:[3,5)\to (0,+\infty)$, a probability measure $\nu_2$ on $[3,5)$ and positive constants $C_2,\lambda_2,\gamma_2>0$ such that, for all $\alpha>0$ and for all probability measure $\mu$ on $[3,5)$,
	\begin{align}
	\label{eq:QSDalpha}
	\left\|e^{\alpha \lambda_2 t}\mathbb P_\mu\left(X^\alpha_t\in\cdot\cap[3,5)\right)-\mu(\eta_2)\nu_2(\cdot)\right\|_{TV}\leq C_2\,e^{-\alpha\gamma_2 t},\ \forall t\geq 0.
	\end{align}
	There exist a positive function $\eta_1:(0,2)\to (0,+\infty)$,  a probability measure $\nu_1$ on $(0,2)$ and  positive constants $C_1,\lambda_1,\gamma_1>0$ such that, for all probability measure $\mu$ on $(0,2)$ and all $\alpha>0$,
	\begin{align}
	\label{eq:QSD0}
	\left\|e^{\lambda_1 t}\mathbb P_\mu\left(X^\alpha_t\in\cdot\cap (0,2)\right)-\mu(\eta_1)\nu_1(\cdot)\right\|_{TV}\leq C_1\,e^{-\gamma_1 t},\ \forall t\geq 0.
	\end{align}
\end{proposition}

\begin{proof}[Proof of Proposition~\ref{prop:QSDs}]
	For any $\alpha>0$, on the event $X^\alpha_0\in[3,5)$, the process remains almost surely in $[3,5)$ until it reaches
        $\partial D$ at the end point $5$. It is known (this can be proved for instance using
        Section 4.5 of~\cite{ChampagnatVillemonais2017a}) that, considering the process $X^\alpha$ restricted to $[3,5)$ absorbed when it reaches $5$, there exists  a probability measure $\mu_\alpha$ on $[3,5)$ (the quasi-stationary distribution of $X^\alpha$ restricted to $[3,5)$), a positive function $\zeta_\alpha:[3,5)\to (0,+\infty)$, and positive constants $c_\alpha,\delta_\alpha,\delta'_\alpha>0$ such that, for all probability measure $\mu$ on $[3,5)$,
	\begin{align}
	\label{eq:QSDalphaproof}
	\left\|e^{\delta_\alpha t}\mathbb P_\mu\left(X^\alpha_t\in\cdot\right)-\mu(\zeta_\alpha)\mu_\alpha(\cdot)\right\|_{TV}\leq c_\alpha\,e^{-\delta'_\alpha t},\ \forall t\geq 0.
	\end{align}
	Also, since $Law((X_t^\alpha)_{t\geq 0})=Law((X^1_{\alpha t})_{t\geq 0})$, we deduce that $\zeta_\alpha=\zeta_1$, $\mu_\alpha=\mu_1$ and $\delta_\alpha=\alpha\delta_1$ for all $\alpha>0$. Moreover, on can take $c_\alpha=c_1$ and $\delta'_\alpha=\alpha\delta'_1$. Setting $\eta_2=\zeta_2$, $\nu_2=\mu_1$, $C_2=c_1$, $\lambda_2=\delta_1$ and $\gamma_2=\delta'_1$, this proves~\eqref{eq:QSDalpha}.
	%Note that, when $X^\alpha_0\in[3,5)$, the law of $X^\alpha$ absorbed at $5$ and the law of $X^\alpha$  absorbed at $0$ and~$5$ are the same, so that $\nu_\alpha$ is also a quasi-stationary distribution for $X^\alpha$ absorbed at $0$ and~$5$.

	Similarly, the law of the process $X^\alpha$ restricted to $(0,2)$ and absorbed when it reaches $\{0,2\}$ does not depend on $\alpha$, and there exists  a probability measure $\nu_1$ on $(0,2)$ (the quasi-stationary distribution of the process $X^\alpha$ conditioned to remain in $(0,2)$),  a positive function $\eta_1:(0,2)\to (0,+\infty)$ and  positive constants $C_1,\lambda_1,\gamma_1>0$ such that, for all probability measure $\mu$ on $(0,2)$ and all $\alpha>0$,
	\begin{align}
	\label{eq:QSD0proof}
	\left\|e^{\lambda_1 t}\mathbb P_\mu\left(X^\alpha_t\in\cdot,\,t<T_0\wedge T_2\right)-\mu(\eta_1)\mu_1(\cdot)\right\|_{TV}\leq C_1\,e^{-\gamma_1 t},\ \forall t\geq 0,
	\end{align}
	where $T_a$ denotes the first hitting time of $\{a\}$ (see again Section~4.5 in~\cite{ChampagnatVillemonais2017a}). Since the process cannot enter $(0,2)$ after time $T_0\wedge T_2$, we deduce that $\mathbb P_\mu\left(X^\alpha_t\in\cdot,\,t<T_0\wedge T_2\right)=\mathbb P_\mu\left(X^\alpha_t\in\cdot\cap (0,2)\right)$.
	%Note that $\mu_0$ is not a quasi-stationary distribution for $X^\alpha$ absorbed at $0$ and $5$. 
	This concludes the proof of Proposition~\ref{prop:QSDs}.
\end{proof}

We consider now the behaviour of the process with initial position in $[2,3)$. For any $\alpha > 0$ and $x\in [2,3)$, denote by $(f^\alpha_t(x))_{t\geq 0}$ the solution of the ODE $\partial f^\alpha_t(x)/\partial t = \psi^1(f^\alpha_t(x)) + \alpha \psi^2(f^\alpha_t(x))$ and $f^\alpha_0(x)=x$. For all $x\in [2,3)$, denote by $t_3(x)=\inf\{ t\geq 0,\ f^\alpha_t(x)=3\}$, then
\begin{align*}
\mathbb P_x(X^\alpha_{t}\in  A)=
\begin{cases}
\mathbf1_{f^\alpha_t(x)\in  A}&\text{ if }f^\alpha_t(x) < 3,\\
\mathbb P_{3}\left(X^\alpha_{t-t_3(x)}\in  A\right)&\text{ if }f^\alpha_t(x)\geq 3,
\end{cases}
\end{align*}
where we used the strong Markov property at time $T_3$  and the fact that $X^\alpha$ is deterministic with drift equal to $\psi^1 + \alpha \psi^2$ on $[2,3]$ (so that $T_3=t_3(x)$ $\P_x$-almost surely). In particular, we deduce from Proposition~\ref{prop:QSDs} that, for all $x\in [2,3)$ and $t\geq t_3(x)$, %$t\geq 1/1 \land \alpha$,
\begin{align*}
\left\|e^{\alpha \lambda_2 (t-t_3(x))}\mathbb P_x\left(X^\alpha_t\in\cdot\right)-\eta_2(3)\nu_2(\cdot)\right\|_{TV}\leq C_2\,e^{-\alpha\gamma_2 (t-t_3(x))}
\end{align*}
and hence%, for all $x\in [2,3)$,
\begin{align*}
\left\|e^{\alpha \lambda_2 t}\mathbb P_x\left(X^\alpha_t\in\cdot\right)-\eta_\alpha(x)\nu_2(\cdot)\right\|_{TV}\leq e^{\alpha(\lambda_2+\gamma_2)t_3(x)}C_2\,e^{-\alpha\gamma_2 t}\leq C'_2 \,e^{-\alpha\gamma_2 t},
\end{align*}
where  $\eta_\alpha(x):=e^{\alpha\lambda_2 t_3(x)}{\eta_2(3)}$ and $C'_2=e^{\alpha(\lambda_2+\gamma_2)t_3(2)}C_2$. By integration of the last inequality and by~\eqref{eq:QSDalpha}, we thus proved that, for any initial distribution in $[2,5)$,
\begin{align}
\label{eq:QSDalphabis}
\left\|e^{\alpha \lambda_2 t}\mathbb P_\mu\left(X^\alpha_t\in\cdot\right)-\mu(\eta_\alpha)\nu_2(\cdot)\right\|_{TV}\leq C'_2\,e^{-\alpha\gamma_2 t},\ \forall t\geq 0,
\end{align}
where
\[
\eta_\alpha(x)=\begin{cases}
e^{\alpha \lambda_2 t_3(x)}{\eta_2(3)}&\text{ if }x\in[2,3]\\
\eta_2(x)&\text{ if }x\in[3,5).
\end{cases}
\]

We are now in a position to apply the results of Section~\ref{sec:abstractQSD}, with $D_1=(0,2)$ and $D_2=[2,5)$.

\paragraph{Case $\alpha<\lambda_1/\lambda_2$.} In this case we observe that Assumption~QSD2 is satisfied (of course with $\alpha\lambda_2$ instead of $\lambda_2$). Indeed, on the one hand~\eqref{eq:eta2def} is immediate from~\eqref{eq:QSDalphabis}, while, for all $x\in D_1$,
\begin{align*}
\sup_{x\in D_1}e^{\alpha\lambda_2 t}\mathbb P_x(X_t\in D_1)\leq e^{(\alpha\lambda_2-\lambda_1 )t}\left(C_1+\|\eta_1\|_\infty\right)\xrightarrow[t\to+\infty]{} 0
\end{align*}
and is $L^1(\mathbb{R}_+)$, where we used~\eqref{eq:QSD0}.
Moreover, we have $\mathbb P_x(X_1\in D_2)>0$ for all $x\in D_1$ and hence, according to Theorem~\ref{thm-qsd2}, $\nu_2(\cdot\cap D_2)$ is the unique quasi-stationary distribution $\nu$ for the process $X^\alpha$ absorbed at $\{0,5\}$, and~\eqref{eq:thmcase1} implies that, for all probability measure $\mu$ on $D$,
\begin{align*}
\Phi^\alpha_t(\mu)\xrightarrow[t\to+\infty]{TV} \nu_2(\cdot \cap D_2).
\end{align*}

\paragraph{Case $\alpha=\lambda_1/\lambda_2$.} In this case, we observe that
Assumption~QSD1-2 is satisfied. To deduce the convergence of $\Phi_t^\alpha(\mu)$, we need to distinguish two cases:
  either $\mu(D_1)=0$ and then the fact that
\begin{align*}
\Phi^\alpha_t(\mu)\xrightarrow[t\to+\infty]{TV} \nu_2(\cdot \cap D_2).
\end{align*}
follows from~\eqref{eq:eta2def}, or $\mu(D_1)>0$ and then it follows from~\eqref{eq:eta1def}, \eqref{eq:eta2def} and~\eqref{eq:thmcase3} that, on the one
hand,
\[
\frac{e^{\lambda_0 t}}{t}\mathbb{P}_\mu(X_t\neq\d)=\frac{e^{\lambda_0 t}}{t}\mathbb{P}_\mu(X_t\in D_1)+\frac{e^{\lambda_0
    t}}{t}\mathbb{P}_\mu(X_t\in D_2, X_t\neq\d)\xrightarrow[t\to+\infty]{}\mu(\eta)>0
\]
and on the other hand, that for all $A\subset D_1\cup D_2$ measurable,
\[
\frac{e^{\lambda_0 t}}{t}\mathbb{P}_\mu(X_t\in A)=\frac{e^{\lambda_0 t}}{t}\mathbb{P}_\mu(X_t\in A\cap D_1)+\frac{e^{\lambda_0
    t}}{t}\mathbb{P}_\mu(X_t\in A\cap D_2)\xrightarrow[t\to+\infty]{}\mu(\eta)\nu_2(A\cap D_2),
\]
where the convergence is uniform with respect to $A$. Hence, in all cases,
\begin{align*}
\Phi^\alpha_t(\mu)\xrightarrow[t\to+\infty]{TV} \nu_2(\cdot \cap D_2).
\end{align*}

\paragraph{Case $\alpha>\lambda_1/\lambda_2$.} In this case, we observe that Assumption~QSD1 is satisfied. Indeed, on the one hand~\eqref{eq:eta1def} holds true, while, for all $x\in D_2$, 
\begin{align*}
\sum_{k \in \Z_+}\sup_{x\in D_2}e^{\lambda_1 k}\mathbb P_x(X_k\in D_2)\leq \sum_{k \in \Z_+} e^{(\lambda_1-\alpha\lambda_2) k}\left(C'_2+\|\eta_\alpha\|_\infty \right) < + \infty,
\end{align*}
where we used~\eqref{eq:QSDalphabis} and $\alpha\lambda_2>\lambda_1$. Then,~\eqref{eq:thmcase2} in Theorem~\ref{thm-qsd1} implies that there exists a probability measure $\nu$ on $D$ such that $\nu(D_1)>0$ and such that, for all probability measure $\mu$ on $D$ such that $\mu(D_1)>0$,
\begin{align*}
\Phi^\alpha_t(\mu)\xrightarrow[t\to+\infty]{TV} \nu.
\end{align*}
In addition,~\eqref{eq:QSDalphabis} entails that, for all probability measure $\mu$ on $D$ such that $\mu(D_1)=0$,
\begin{align*}
\Phi^\alpha_t(\mu)\xrightarrow[t\to+\infty]{TV} \nu_2(\cdot\cap D_2).
\end{align*}
This concludes the proof of Theorem~\ref{thm:main}.

\bibliography{biblio}
\bibliographystyle{amsplain}

\end{document}